\documentclass[11pt]{article}
\usepackage{amsmath}
\usepackage{amsfonts}
\usepackage{amssymb}
\usepackage{amsthm}
\usepackage{palatino}
\usepackage{MnSymbol, stmaryrd} 

\renewcommand\footnotemark{}

\usepackage[margin=2.5cm, vmargin={1.5cm},includefoot]{geometry}

\begin{document}

\def\s#1#2{\langle \,#1 , #2 \,\rangle}

\def\H{{\mathbb H}}
\def\F{{\mathbb F}}
\def\C{{\mathbb C}}
\def\R{{\mathbb R}}
\def\Z{{\mathbb Z}}
\def\Q{{\mathbb Q}}
\def\N{{\mathbb N}}
\def\st{{\mathbb S}}
\def\D{{\mathbb D}}
\def\B{{\mathbb B}}
\def\G{{\Gamma}}
\def\GH{{\G \backslash \H}}
\def\g{{\gamma}}
\def\L{{\Lambda}}
\def\ee{{\varepsilon}}
\def\K{{\mathcal K}}
\def\Re{\text{\rm Re}}
\def\Im{\text{\rm Im}}
\def\SL{\text{\rm SL}}
\def\GL{\text{\rm GL}}

\def\slz{\text{\rm SL}(2,\Z)}
\def\slr{\text{\rm SL}(2,\R)}

\def\sgn{\text{\rm sgn}}
\def\tr{\text{\rm tr}}
\def\ca{{\mathfrak a}}
\def\cb{{\mathfrak b}}
\def\cc{{\mathfrak c}}
\def\cd{{\mathfrak d}}
\def\ci{{\infty}}

\def\sa{{\sigma_\mathfrak a}}
\def\sb{{\sigma_\mathfrak b}}
\def\sc{{\sigma_\mathfrak c}}
\def\sd{{\sigma_\mathfrak d}}
\def\si{{\sigma_\infty}}

\def\se{{\sigma_\eta}}
\def\sz{{\sigma_{z_0}}}

\def\sai{{\sigma^{-1}_\mathfrak a}}
\def\sbi{{\sigma^{-1}_\mathfrak b}}
\def\sci{{\sigma^{-1}_\mathfrak c}}
\def\sdi{{\sigma^{-1}_\mathfrak d}}
\def\sii{{\sigma^{-1}_\infty}}
\def\PSL{\text{\rm PSL}}
\def\vol{\text{\rm vol}}
\def\I{\text{\rm Im}}
\def \Im { {\mathrm{Im}} }
\def \Re { {\mathrm{Re}} }
\def \P {{\mathbb P}}
\def \re { {\mathrm{Re}} }

\def\e{\mathbf{e}}
\def\er{\eqref}
\def\lqs{\leqslant}
\def\gqs{\geqslant}

\newcommand{\m}[4]{\begin{pmatrix}#1&#2\\#3&#4\end{pmatrix}}
\newcommand{\n}{\frac1{\sqrt{37}}}
\newcommand{\ms}[4]{\left(\smallmatrix #1&#2\\#3&#4\endsmallmatrix\right)}
\newcommand{\ns}{\textstyle\frac1{\sqrt{37}}}

\newtheorem{theorem}{Theorem}[section]
\newtheorem{lemma}[theorem]{Lemma}
\newtheorem{prop}[theorem]{Proposition}
\newtheorem{cor}[theorem]{Corollary}
\newtheorem{conj}[theorem]{Conjecture}
\newtheorem{defn}[theorem]{Definition}
\newtheorem{remark}[theorem]{Remark}
\newtheorem{example}[theorem]{Example}

\renewcommand{\labelenumi}{(\roman{enumi})}

\numberwithin{equation}{section}
\newtheorem{main-theorem}{Theorem}
\newtheorem{main-prop}[main-theorem]{Proposition}

\bibliographystyle{alpha}


\title{Noncommutative modular symbols\\
and Eisenstein series}

\author{
  Gautam ~Chinta,
  Ivan ~Horozov,
  Cormac ~O'Sullivan\footnote{{\it Date:} Aug 27, 2018.
  \newline \indent \ \ \
  {\it 2010 Mathematics Subject Classification:} 11F11, 11F67, 30F35.
  \newline \indent \ \ \
  {\it Key words and phrases:} modular symbols, iterated integrals, Eisenstein series.}
  }







\date{}

\maketitle

\begin{abstract}
We form real-analytic Eisenstein series twisted by Manin's noncommutative modular symbols. After developing their basic properties, these series are shown to have meromorphic continuations to the entire complex plane and satisfy functional equations in some cases. This theory neatly contains and generalizes earlier work in the literature on the properties of Eisenstein series twisted by classical modular symbols.
\end{abstract}




\section{Introduction}
\subsection{Background}
Let $\Gamma$ be a congruence subgroup acting on the complex upper
half plane $\H$ together with its cusps: $\overline{\H}=\H \cup
\P^1(\Q)$.  The quotient is the modular curve
$X=\Gamma\backslash\overline{\H}.$  Manin  \cite{M1} introduced
{\em modular symbols}, which are elements of $H_1(X,\Q).$ There is a
natural pairing between modular symbols and holomorphic weight two
cusp forms on $\Gamma$, obtained by integration.  This pairing,
combined with the action of Hecke operators, yields rich arithmetic
information about periods of elliptic curves, special values of
$L$-functions and Fourier coefficients of modular forms.

In   \cite{Go} Goldfeld began a program of studying the distribution of modular symbols.
For $\g \in \G$ we use the notation
\begin{equation}\label{mod}
    \s{\g}{f} := 2\pi i\int_{z_0}^{\g z_0} f(u) \, du
\end{equation}
for the pairing between  the modular symbol $\{z_0,\g z_0\}$ and the weight two cusp form $f$ for $\G$.
It may be seen that \er{mod} is independent of $z_0 \in \overline{\H}$, and since $f$ is a cusp form, $\langle \gamma,f\rangle$
depends only on the bottom row of the matrix $\gamma$.
Following earlier  work of  O'Sullivan \cite{OS} and
Petridis \cite{Pe}, it is shown in Petridis-Risager \cite[Thms. A,G]{PeRi} that as $T\to\infty$ the values in
\begin{equation*}
  \left\{
\frac{\Re \s{\ms{*}{*}{c}{d}}{f}}
{\sqrt{\log(c^2+d^2)}}
\ : \ c^2+d^2 \lqs T
\right\}
\end{equation*}
follow a normal distribution, and  that
\begin{equation*}
  \sum_{c^2+d^2 \lqs T} \left| \s{\ms{*}{*}{c}{d}}{f}\right|^2 =\frac{16 \pi^2}{\vol(\GH)} || f||^2 T \log T +O(T) \quad \text{as} \quad T\to \infty.
\end{equation*}
Very recently, Petridis and Risager in \cite{PR18} have also verified, exactly and on average, some of the conjectures of Mazur, Rubin and Stein
on the statistics of modular symbols.

These results depend on  properties
of {\em Eisenstein series twisted by modular
  symbols}, that is, functions of the form
\begin{equation}
  \label{def:ems}
  E(z,s;f)=\sum_{\gamma \in \Gamma_\infty\backslash\Gamma}
\langle\gamma,f\rangle\Im(\gamma z)^{s}
\end{equation}
and their generalizations, as in
\cite{Go}.  Here $z$ is in the upper half plane and $s$ is a complex
number, initially with real part bigger than $1$ to ensure
convergence.  The function $E(z,s;f)$ is not automorphic in $z$, but
it does satisfy the more complicated relation
\begin{equation}
  \label{eq:1}
  E(\gamma_1\gamma_2z,s;f)- E(\gamma_1z,s;f)- E(\gamma_2z,s;f)+ E(z,s;f)=0
\end{equation}
for all $\gamma_1,\gamma_2\in \Gamma.$

Define
 $ (f|_k \g)(z)$ as $j(\g,z)^{-k}f(\g z)$ for
 $j( \left(\smallmatrix a & b \\ c & d
  \endsmallmatrix\right),z):=cz+d$ and extend this action to the group
ring $\C[\G]$ by linearity.  The relation (\ref{eq:1}) led
Chinta-Diamantis-O'Sullivan \cite{cdos} to define a \emph{second-order
  modular form} on $\Gamma$ of weight $k$ as a holomorphic function on
the upper half plane which satisfies
\begin{equation}
  \label{def:somf}
  f|_k(\gamma_1-I)(\gamma_2-I)=0, \mbox{ \ for all \ }
  \gamma_1,\gamma_2\in \Gamma.
\end{equation}
 An \emph{$n$th-order modular
  form} of  weight $k$ on $\Gamma$ satisfies
 \begin{equation}
  \label{def:homf}
  f|_k(\gamma_1-I)(\gamma_2-I) \cdots (\gamma_n-I)=0,
\mbox{ \ for all \ }
  \gamma_i\in \Gamma.
\end{equation}
Higher-order modular forms were also independently defined by Kleban
and Zagier \cite{kz}.  Their motivation comes from the study of
modular properties in crossing probabilities on 2 dimensional
lattices.
One can similarly define higher-order (real-analytic) Eisenstein
series generalizing the second-order series in (\ref{def:ems}) as in \cite{PeRi} and \cite{JOS08}.  Very
general series of this form are studied by Diamantis and Sim
\cite{DS}.

Manin \cite{M2} has initiated the development of a theory of {\em
  noncommutative modular symbols}.  These arise as iterated integrals
of cusp forms and Eisenstein series. The principal motivation for
their study comes from {\em multiple zeta values}.  See, for example,
the paper of Choie and Ihara \cite{ChoieIhara} which gives explicit
formulas for iterated integrals in terms of multiple Hecke
$L$-functions.  Generalizing in another direction, Horozov \cite{Ho15} defines
{\em noncommutative Hilbert modular symbols}.


In the present paper we construct  Eisenstein series twisted by Manin's
noncommutative modular symbols. The result is a generating series
whose coefficients contain the higher-order
Eisenstein series studied in \cite{PeRi,JOS08,DS} as well as further new series not studied before. We find that the automorphic
properties and functional equations of the earlier higher-order Eisenstein series are elegantly and
succinctly encapsulated in our new formalism.

\subsection{Main results}\label{sec:mainresults}

We introduce some basic notation in order to state our main results;
see e.g. Iwaniec \cite{IwSp} for a fuller background.  The group
$\SL_2(\R)$ acts by linear fractional transformations on
$\H \cup \P^1(\R)$.  Let $\G \subset \SL_2(\R)$ be a Fuchsian group of
the first kind, i.e. a discrete subgroup of $\SL_2(\R)$ with quotient
$\G\backslash\H$ of finite hyperbolic volume. We may fix a finite set
of inequivalent cusps for $\G$, which we assume to be nonempty.  Let
$\G_\ca$ be the subgroup of $\G$ fixing the cusp $\ca$. Then
$\overline{\G}_\ca$ is isomorphic to $\Z$, where the bar means the
image under the map $\SL_2(\R) \to \SL_2(\R)/\pm I$. (We don't assume
$-I \in \G$.) This isomorphism can be seen explicitly as there exists
a scaling matrix $\sa \in \SL_2(\R)$ such that $\sa \infty = \ca$ and
$$
 \sai\overline{\G}_\ca \sa  = \left\{\left. \pm \ms{1}{\ell}{0}{1} \
   \right|\  \ell\in \Z\right\}.
 $$

 A  function $f$ on $\H$ has weight $k$ for $\G$ if $f$ satisfies
 $ (f|_k \g)(z) = f(z)$ for all $\g \in \G$. If $f$ is also holomorphic on $\H$
and has a Fourier expansion at each cusp $\ca$ of the form
\begin{equation}\label{expnpar}
\left(f|_k  \sa\right) (z)=\sum_{m =0}^\infty c_\ca(m;f) e^{2\pi i m z}
\end{equation}
then $f$ is a modular form  for $\G$.
If $c_\ca(0;f)=0$ for every cusp $\ca$ then $f$ is called a cusp form.

For $a$ and $b$  in $\H\cup\{\mathtt{cusps}\}$, set
\begin{equation} \label{cab}
  C_a^b(f_1, f_2, \dots,  f_n)
  :=  \int_a^b f_{n}(z_n) \int_a^{z_n} f_{n-1}(z_{n-1}) \cdots \int_a^{z_2} f_1(z_1) \, dz_1 \cdots dz_n
\end{equation}
where $f_1, f_2, \dots, f_n$  are  weight $2$ holomorphic cusp forms for $\G$. Near end-points
that are cusps, we assume the path of integration follows a geodesic
to ensure we remain in the region where the cusp forms have exponential
decay. In this way \er{cab} is always absolutely convergent.

Let $\mathbf{f}=(f_1,f_2, \dots, f_r)$ be an $r$-tuple of not necessarily distinct, weight $2$
holomorphic cusp forms for $\G$.  To each $f_i$ we associate the
variable $X_i$ and define
 \begin{equation} \label{iab}
   I_a^b=I_a^b(\mathbf{f}) :=
   1 + \sum_i  C_a^b(f_i) \cdot X_i  +
   \sum_{i,j} C_a^b(f_i, f_j) \cdot X_i X_j  + \sum_{i,j,k} C_a^b(f_i, f_j, f_k) \cdot X_i X_j X_k +\cdots.
\end{equation}
Thus $I_a^b$ is an element of the ring
$\C\llangle X_1,\ldots, X_r\rrangle$ of formal power series in the
non-commuting variables $X_i$ with coefficients in $\C$.  We will also use the notation $I_a^b(f_1, f_2, \dots,  f_r)$ for \er{iab}. In the case when $r=0$ it is convenient to set
\begin{equation}\label{icd}
  I_a^b():=1, \qquad C_a^b():=1.
\end{equation}

As we will see, these symbols satisfy the two key relations of {\em concatenation} and {\em $\G$-invariance}:
\begin{align}\label{msc}
 I_a^b I_b^c & = I_a^c,\\
 I_{\g a}^{\g b} & = I_a^b  \label{msg}
\end{align}
for all $a, b, c \in \H\cup\{\mathtt{cusps}\}$ and all $\g \in \G$. Manin's {\em noncommutative modular symbol} is $I_{\g a}^a$ for $\g \in \G$ and base point $a \in \H\cup\{\mathtt{cusps}\}$.
As described in \cite[Prop. 2.5.1]{M2}, $I_{\g a}^a$ is a representative of a nonabelian cohomology class in $H^1(\G,\Pi)$ where $\Pi$ is the group of invertible elements of $\C\llangle X_1,\ldots, X_n\rrangle$ with constant term $1$.

For each cusp $\cb$ we define the {\em Eisenstein series with noncommutative twists} as
\begin{equation} \label{def:etbncms} \mathcal{E}_{\cb}(z,s):=\sum_{\g
    \in \G_\cb\backslash \G} I_{\g a}^{a} \cdot \Im(\sbi\g z)^s.
\end{equation}
Our first result proves the basic properties of this  series. Recall that the hyperbolic Laplacian $\Delta$ is given by $y^2(d^2/dx^2+d^2/dy^2)$ for $z=x+iy$.

\begin{theorem}\label{thm:intro}
  The series $\mathcal{E}_{\cb}(z,s)$ is well-defined and absolutely convergent for
  $\re(s)>1$. It is an eigenfunction of the Laplacian:
  \begin{equation}\label{lap}
    \Delta \mathcal{E}_{\cb}(z,s) = s(s-1) \mathcal{E}_{\cb}(z,s)
  \end{equation}
   and for all $\delta \in \G$ satisfies
   \begin{equation}\label{ginv}
     \mathcal{E}_\cb(\delta z,s) = I_{a}^{\delta a} \cdot
  \mathcal{E}_\cb(z,s).
   \end{equation}
\end{theorem}

In the above statement, the absolute convergence of
$\mathcal{E}_{\cb}(z,s)$ means that the coefficient of each
$X_{i_1}\cdots X_{i_n}$ in the formal power series defining
$\mathcal{E}_{\cb}(z,s)$ is an absolutely convergent sum over
$\gamma \in \G_\cb\backslash \G.$ The meromorphic continuation of
$\mathcal{E}_{\cb}(z,s)$ described below has the same meaning.

Note that changing the base point in \er{def:etbncms} has a simple
effect: if $\mathcal{E}_{\cb}'(z,s)$ is defined as in
\er{def:etbncms}, but with base point $c$ instead of $a$, then
$\mathcal{E}_{\cb}'(z,s) = I_{c}^{a} \cdot \mathcal{E}_{\cb}(z,s)
\cdot I_{a}^{c}$ since
\begin{equation}\label{bpac}
  I_{\g c}^{c} = I_{\g c}^{\g a} \cdot I_{\g a}^{a}  \cdot I_{a}^{c}  = I_{c}^{a} \cdot I_{\g a}^{a}  \cdot I_{a}^{c}
\end{equation}
by \er{msc} and \er{msg}.

\begin{theorem}\label{thm:con}
  The series $\mathcal{E}_{\cb}(z,s)$ admits a meromorphic continuation to all $s \in \C$.
\end{theorem}

Theorem \ref{thm:con} is shown in Section \ref{contf}. In fact much more is proved there, including bounds on the continued series $\mathcal{E}_{\cb}(z,s)$ as well as its Fourier coefficients.

When $r=1$, so there is just one cusp form $f_1=f$ and $I_a^b$ is in
the commutative power series ring $C\llbracket X\rrbracket$, the
function $\mathcal{E}_{\cb}(z,s)$ makes sense not just as a formal
power series but also as a convergent series in the real parameter
$X$.  In fact, this is the approach taken by Petridis \cite{Pe} and
Petridis-Risager \cite{PeRi} who use spectral deformation theory to
study generalizations of the series (\ref{def:ems}).  Thus our
work may be considered an extension of their ideas in using a
generating series to simultaneously study the properties of a
collection of modular symbols or iterated integrals.  See also our
remarks following Theorems \ref{cont2} and \ref{thm:fe}.



\subsection*{Acknowledgments}

We thank the organizers of {\em Building Bridges: $3^{rd}$ EU/US
  Summer School and Workshop on Automorphic Forms and Related Topics}
for encouraging us to prepare and submit this work. We also thank Nikolaos Diamantis and Yuri Manin for their comments on an earlier draft. All three authors were partially supported by PSC-CUNY Awards and the first named
author  acknowledges support  from NSF-DMS
1601289.

\section{Properties of iterated integrals}\label{sec:iterated_integrals}
\subsection{Basic properties}
We collect in this section properties of iterated integrals which we
will need in our study of noncommutative modular symbols and the
associated Eisenstein series.  Note that  the
definition (\ref{cab}) makes sense for general holomorphic functions
$f_1, \ldots, f_n$ on $\C$ and endpoints  $a,$ $b\in\C$. So, just in this subsection, we include these more general functions $f_i$.

\begin{prop}\label{prop:ii_properties1}  Let
  $a,b,c\in\C$ and let $f_1, \ldots ,f_n$ be
  holomorphic functions on $\C.$ Then we have
  \begin{enumerate}
  \item 
     $\displaystyle C_a^{b}(\overbrace{f_1,f_1, \dots,  f_1}^n)=
      \frac{1}{n!}\biggl( \int^{b}_a f_1(u) \, du \biggr)^n$,
    \item  $C_a^{b}(f_1, f_2, \dots,  f_n)= (-1)^n  C_b^{a}(f_n,
      f_{n-1}, \dots,  f_1)$,
    \item 
 $ C_a^c(f_1, f_2, \dots,  f_n) = \sum_{0 \lqs i \lqs n} C_a^{b}(f_1,  \dots,  f_i) \cdot C_b^{c}(f_{i+1},  \dots,  f_n).$
  \end{enumerate}
 \end{prop}
\begin{proof}
  Part (i) follows by induction and the change of variables
  $w= \int^{z}_a f(u) \, du$.   Part (ii) is \cite[Lemma 1.1, (v)]{DH}.
We may also give an elementary proof by induction as follows. The cases $n=0,1$ are true, so assume $n \gqs 2$. Set
  \begin{equation}
    \label{eq:Fj}
  F_j(z):=\int_{a_0}^z f_j(w)\, dw
  \end{equation}
for some fixed $a_0$. Then
\begin{equation}\label{fjz}
  C_a^{b}(f_1, f_2, \dots,  f_n) = C_a^{b}(F_1\cdot f_2,f_3, \dots,  f_n) - F_1(a) C_a^{b}(f_2, \dots,  f_n).
\end{equation}
The induction hypothesis yields
\begin{equation} \label{fjz2}
  C_a^{b}(f_1, f_2, \dots,  f_n) = (-1)^{(n-1)}\left(C_b^{a}(f_n, \dots, f_2 \cdot F_1) - F_1(a) C_b^{a}(f_n, \dots,  f_2)\right).
\end{equation}
Applying \er{fjz} to each term on the right of \er{fjz2} and using
induction again shows
\begin{multline}\label{fzc}
  C_a^{b}(f_1, f_2, \dots,  f_n) = -C_a^{b}(F_1 \cdot f_2,f_3, \dots,  f_{n-2}, f_{n-1}\cdot F_n)
  + F_n(b) C_a^{b}(F_1 \cdot f_2,f_3, \dots,  f_n)\\
  + F_1(a) C_a^{b}(f_1, \dots,  f_{n-2}, f_{n-1}\cdot F_n)
  -F_1(a) F_n(b) C_a^{b}(f_2, \dots,  f_{n-1}).
\end{multline}
Expressing $C_b^{a}(f_n, f_{n-1}, \dots,  f_1)$ with formula \er{fzc}
and using the induction hypothesis completes the proof of (ii).

Part (iii) is \cite[Lemma 1.1, (iii)]{DH}. We may easily prove it by taking $a_0=a$ in the definition (\ref{eq:Fj}) of
$F_1$ to get
\begin{equation}\label{wer2}
  C_a^c(f_1, f_2, \dots,  f_n) = C_a^c(F_1 \cdot f_2,f_3, \dots,  f_n).
\end{equation}
Applying \er{wer2} to both sides of (iii) and
using induction completes the proof.
\end{proof}

Proposition \ref{prop:ii_properties1} (iii) implies the concatenation
relation \er{msc}.  With Proposition \ref{prop:ii_properties1} (i) we
have, at least formally,
\begin{equation*}
  I_a^b(f_1) = 1 +   C_a^b(f_1) \cdot X +  C_a^b(f_1, f_1) \cdot X^2  +  C_a^b(f_1, f_1,f_1) \cdot X^3  +\cdots = \exp\biggl(  \int^{b}_a f_1(u) \, du \cdot X \biggr).
\end{equation*}
We will see that the iterated integrals $C_a^{b}(f_1, f_2, \dots,  f_n)$ satisfy a further family of general identities in \er{shuffle}.

\subsection{Modular properties}
 The iterated integrals will have
certain modular invariance properties when $f_1, \dots, f_n$ are  weight $2$ cusp forms, as we assume for the rest of the paper.
The $\G$-invariance relation \er{msg} is equivalent to
\begin{equation*}
  C_{\g a}^{\g b}(f_1, f_2, \dots,  f_n)=C_a^b(f_1, f_2, \dots,  f_n)
\end{equation*}
which may be easily verified directly.

\begin{prop} \label{cab0}
For any $a\in \H$ and any parabolic $\g \in \G$ we have $C_a^{\g a}(f_1, f_2, \dots,  f_n)=0$ when $n\gqs 1$.
\end{prop}
\begin{proof}
 Suppose $\g$ fixes the cusp $\cb$. Then $\sbi \g \sb = \pm \ms{1}{\ell}{0}{1}$ for some $\ell \in \Z$. We first consider $C_a^b$ for $b\in \H$. Using the change of variables $z_i=\sb w_i$ in \er{cab} yields
 \begin{align*}
   C_a^{b}(f_1, f_2, \dots,  f_n) &= \int_{\sbi a}^{\sbi b} f_{n}(\sb w_n) \int_{\sbi a}^{w_n} f_{n-1}(\sb w_{n-1}) \cdots \int_{\sbi a}^{w_2} f_1(\sb w_1) \, d\sb w_1 \cdots d\sb w_n \\
   & = \int_{\sbi a}^{\sbi b} (f_{n}|_2 \sb)( w_n) \int_{\sbi a}^{w_n} (f_{n-1}|_2 \sb)( w_{n-1})  \cdots \int_{\sbi a}^{w_2} (f_{1}|_2 \sb)( w_1)  \, dw_1 \cdots dw_n.
 \end{align*}
 Writing $u=\sbi a$, $v=\sbi b$, $\e(z):=e^{2\pi i z}$ and using \er{expnpar} shows
 \begin{multline} \label{intx}
   C_a^{b}(f_1, f_2, \dots,  f_n) =
   \sum_{m_1, \dots, m_n \in \Z_{\gqs 1}} c_\cb(m_1;f_1) \cdots c_\cb(m_n;f_n)\\
   \times
   \int_u^v \e(m_n w_n)\int_u^{w_n}\e(m_{n-1}w_{n-1}) \cdots \int_u^{w_2} \e(m_1 w_1) \, dw_1 \cdots dw_n.
 \end{multline}
 Note that \er{intx} is absolutely convergent because the coefficients $c_\cb(m_i;f_i)$ have polynomial growth in $m_i$ and the integrals have exponential decay when $u,v \in \H$. Taking $b=\g a$ we find
 \begin{equation*}
   v= \sbi \g a = \sbi \g \sb u = u+\ell.
 \end{equation*}
 Since the integrand in \er{intx} is a polynomial in $\e(w_n)$ with constant term zero, it follows that \er{intx} is zero for $b=\g a$ as desired.
\end{proof}

\begin{cor} \label{cr}
For any $a,b \in \H \cup \{\mathtt{cusps}\}$ and any parabolic $\g \in \G$ we have
\begin{gather}\label{ga}
  I_a^{\g b}=I^b_{\g a}=I^b_{a}, \\
  C_a^{\g b}(f_1, f_2, \dots,  f_n) = C^b_{\g a}(f_1, f_2, \dots,  f_n) = C_a^{b}(f_1, f_2, \dots,  f_n). \label{ga2}
\end{gather}
\end{cor}
\begin{proof}
Let $c \in \H$ and we may change to this base point using
\begin{equation} \label{iabc}
  I_a^{\g b} = I_a^c \cdot I_c^{\g c}\cdot  I_c^b
\end{equation}
which follows as in \er{bpac}. Proposition \ref{cab0} implies that $I_c^{\g c}=1$ and hence $I_a^{\g b}=I^b_{a}$. The second identity in \er{ga} is similar and then \er{ga2} follows from \er{ga}.
\end{proof}

\begin{example}{\rm
To give an example of the relation \er{intx}, we  compute
\begin{multline*}
  \int_u^v \e(m_2 w_2)\int_u^{w_2}\e(m_{1}w_{1}) \, dw_1 dw_2 \\
  =
  \frac{1}{(2\pi i)^2}\left( \frac{\e((m_1+m_2)v)-\e((m_1+m_2)u)}{m_1(m_1+m_2)}
  + \frac{\e((m_1+m_2)u) - \e(m_1 u+m_2 v)}{m_1 m_2}\right).
\end{multline*}
 With the identity
 \begin{equation*}
   -\frac{1}{m_1(m_1+m_2)}+\frac{1}{m_1 m_2} = \frac{1}{m_2(m_1+m_2)}
 \end{equation*}
 we obtain, with $u=\sbi a$, $v=\sbi b$ as before,
 \begin{equation*}
  C_a^{b}(f_1, f_2)=\sum_{m_1,m_2 \in \Z_{\gqs 1}} \frac{c_\cb(m_1;f_1)  c_\cb(m_2;f_2)}{(2\pi i)^2}
 \left( \frac{\e((m_1+m_2)v)}{m_1(m_1+m_2)}
  + \frac{\e((m_1+m_2)u)}{m_2(m_1+m_2)}- \frac{\e(m_1 u+m_2 v)}{m_1 m_2}\right).
 \end{equation*}}
\end{example}
See \cite{ChoieIhara} for more examples of expressing iterated
integrals as multiple $L$-functions.

\subsection{Growth estimates for iterated
integrals}
Here we give some growth estimates  which will be needed for the convergence of Eisenstein series
formed with noncommutative modular symbols.

\begin{prop} \label{poi} Let $\cb$ be a cusp of $\G$ and let
  $a\in \H \cup \{\mathtt{cusps}\}$. For an implied constant independent
  of $b \in \H$,
\begin{equation} \label{incd}
   C_a^{b}(f_1, f_2, \dots,  f_n) \ll 1+|\log \Im(\sbi b)|^n.
\end{equation}
\end{prop}
\begin{proof}
  Recall that for any cusp form $f$ of weight 2, we have
  $y|f(z)| \ll 1$ for all $z\in \H$. We use induction on $n$ to prove the proposition, the
  case $n=0$ being clear. First consider
  $C_{\sb i}^{b}(f_1, \dots, f_n)$ with imaginary number $i$. There
  exists $\g \in \G_\cb$ such that
  $\Re(\sbi \g \sb i) = \Re(\sbi b)+t$ for $|t|\lqs 1/2$. By
  Corollary \ref{cr},
  $C_{\g \sb i}^{b}(f_1, \dots, f_n) = C_{\sb i}^{b}(f_1, \dots,
  f_n)$. Hence
\begin{align*}
  C_{\sb i}^{b}(f_1,  \dots,  f_n) & = \int_{\g \sb i}^{b} f_n(w) \cdot C_{\g \sb i}^{w}(f_1,  \dots,  f_{n-1}) \, dw\\
   & = \int_{\sbi \g \sb i}^{\sbi b} (f_n|_2 \sb)(z) \cdot C_{\sb i}^{\sb z}(f_1,  \dots,  f_{n-1}) \, dz.
\end{align*}
With our cusp form bound and induction, the integrand is $\ll (1+|\log y|^{n-1})/y$. Hence
\begin{align*}
  C_{\sb i}^{b}(f_1,  \dots,  f_n) & \ll  \int_{i-1/2}^{i+1/2}\frac{1+|\log y|^{n-1}}{y} \, dx + \int_{1}^{\Im(\sbi b)}\frac{1+|\log y|^{n-1}}{y} \, dy\\
   & \ll 1+ |\log \Im(\sbi b)|^n.
\end{align*}
Finally, use the identity
\begin{equation} \label{cabi}
    C_a^{b}(f_1, f_2, \dots,  f_n) = \sum_{0 \lqs j \lqs n} C_a^{\sb i}(f_1,  \dots,  f_j) C_{\sb i}^{b}(f_{j+1},  \dots,  f_n)
\end{equation}
to obtain \er{incd} and complete the induction.
\end{proof}

Note that, with Proposition \ref{prop:ii_properties1} (ii) it is clear that we may replace
$C_a^{b}(\dots) $ by $C_b^{a}( \dots)$ in
the statement of Proposition \ref{poi}.

\begin{cor} \label{ghj}
Let $\cb$ be a cusp of $\G$. For  an implied constant independent of $a,$ $b \in \H$,
\begin{equation*} 
   C_a^{b}(f_1, f_2, \dots,  f_n) \ll 1+|\log \Im(\sbi a)|^n+|\log \Im(\sbi b)|^n.
\end{equation*}
\end{cor}
\begin{proof}
This follows from Proposition \ref{poi} and \er{cabi}.
\end{proof}

\begin{cor} \label{hjk}
Let $\cb$ be a cusp of $\G$ and let $a\in \H \cup \{\mathtt{cusps}\}$. For  an implied constant independent of $\g \in \G$ and $z \in \H$,
\begin{equation*} 
   C_{\g a}^{a}(f_1, f_2, \dots,  f_n) \ll 1+|\log \Im(\sbi z)|^n +|\log \Im(\sbi \g z)|^n.
\end{equation*}
\end{cor}
\begin{proof}
Since $I_{\g a}^{a} =I_{\g a}^{\g z}I_{\g z}^{ z}I_{z}^{a} = I_{ a}^{ z}I_{\g z}^{ z}I_{z}^{a}$, we have
\begin{equation*}
  C_{\g a}^{a}(f_1, f_2, \dots,  f_n) = \sum_{0 \lqs i\lqs j\lqs n} C_a^{z}(f_1,  \dots,  f_i) C_{\g z}^{z}(f_{i+1},  \dots,  f_j) C_z^{a}(f_{j+1},  \dots,  f_n).
\end{equation*}
With Proposition  \ref{poi} and Corollary \ref{ghj}, this is
\begin{multline*}
  \ll \sum_{0 \lqs i\lqs j\lqs n} \left(1+|\log \Im(\sbi z)|^i\right)\left(1+|\log \Im(\sbi z)|^{j-i}+|\log \Im(\sbi \g z)|^{j-i}\right)\\
\times \left(1+|\log \Im(\sbi z)|^{n-j}\right)
\end{multline*}
and the result follows.
\end{proof}

The implied constants in Proposition  \ref{poi} and Corollaries \ref{ghj}, \ref{hjk} may depend on the remaining parameters. For example, the implied constant in Corollary \ref{hjk} depends on $\cb$, $a$, $n$ and $f_1, \dots, f_n$.

\section{Eisenstein series twisted by noncommutative modular symbols}\label{sec:eis_series}

In this section we introduce our primary object of study, the
Eisenstein series formed with noncommutative modular symbols.  The
series we define here are slightly more general than those described
in the introduction.

\subsection{Definition and convergence} \label{31}
Let $\mathbf{f}=(f_1,f_2, \dots, f_r)$ and
$\mathbf{g}=(g_1,g_2, \dots, g_t)$ be lists of weight $2$ cusp
forms. To each $f_i$ we associate the variable $X_i$ and to each $g_i$
we associate $Y_i$. Let $I_a^b=I_a^b(\mathbf{f})$ be as before in
\er{iab} and put
 \begin{equation} \label{jab}
  J_a^b =  J_a^b(\mathbf{g}) = 1 + \sum_i  C_a^b(g_i) \cdot Y_i  + \sum_{i,j} C_a^b(g_i, g_j) \cdot Y_i Y_j  + \cdots.
\end{equation}
Then $I_a^b \cdot \overline{J_a^b}$ is an element of the ring of
formal power series in the variables $X_i$ and $\overline{Y_i}$ with
coefficients in $\C$. We take these variables to be non-commuting,
except for the relations $X_i\overline{Y_j}=\overline{Y_j}X_i$ for
$1\lqs i\lqs r$, $1\lqs j\lqs t$. We may write
\begin{multline} \label{wo}
  I_a^b \cdot \overline{J_a^b} = 1+  \sum_i  C_a^b(f_i) \cdot X_i +  \sum_i  \overline{C_a^b(g_i)} \cdot \overline{Y_i}\\
  + \sum_{i,j} C_a^b(f_i, f_j) \cdot X_i X_j +\sum_{i,j} C_a^b(f_i) \overline{C_a^b(g_j)} \cdot X_i \overline{Y_j}
  +\sum_{i,j} \overline{C_a^b(g_i, g_j)} \cdot \overline{Y_i Y_j} + \cdots .
\end{multline}


Fix a base point $a\in \H \cup \{\mathtt{cusps}\}$.  For $z\in \H$ and
any cusp $\cb$ of $\G$, define the {\em Eisenstein series twisted by
  noncommutative modular symbols} as
\begin{equation} \label{edef}
  \mathcal{E}_{\cb}(z,s):=\sum_{\g \in \G_\cb\backslash \G}  I_{\g a}^{a} \cdot \overline{J_{\g a}^{a}} \cdot \Im(\sbi\g z)^s.
\end{equation}
This agrees with our earlier definition \er{def:etbncms} when
$\mathbf{g}$ is empty.  The summands are well-defined by Corollary
\ref{cr}. Expanding with \er{wo} gives
\begin{multline} \label{wo2}
 \mathcal{E}_{\cb}(z,s) = E_{\cb}(z,s) +  \sum_i  E_{\cb}(z,s; f_i) \cdot X_i +  \sum_i  E_{\cb}(z,s; \overline{g_i}) \cdot \overline{Y_i}\\
  + \sum_{i,j} E_{\cb}(z,s; f_i, f_j) \cdot X_i X_j +\sum_{i,j} E_{\cb}(z,s; f_i, \overline{g_j}) \cdot X_i \overline{Y_j}
  +\sum_{i,j}E_{\cb}(z,s; \overline{g_i}, \overline{g_j})\cdot \overline{Y_i Y_j} + \cdots
\end{multline}
where
\begin{equation}
  E_\cb(z,s;f_{i_1}, \dots ,f_{i_m},\overline{g_{j_1}},\dots,\overline{g_{j_n}})  := \sum_{\g \in \G_\cb\backslash \G} C^a_{\g a}(f_{i_1}, \dots ,f_{i_m}) \overline{C^a_{\g a}(g_{j_1}, \dots ,g_{j_n})} \Im(\sbi \g z)^s. \label{eis2}
\end{equation}


\begin{theorem} \label{thbn}
The series $E_\cb(z,s;f_{1}, \dots ,f_{m},\overline{g_{1}},\dots,\overline{g_{n}})$ from \er{eis2} is absolutely convergent for $\Re(s)>1$. This convergence is uniform for $s$ in compact sets.
\end{theorem}
\begin{proof}
In the case $m=n=0$ the theorem is true for $E_\cb(z,s)$, which is
just the ordinary real-analytic Eisenstein series for $\Gamma.$
By Corollary \ref{hjk}, the summand of the series corresponding to $\g$ is bounded by a constant times
\begin{equation} \label{hyu}
   \left( 1+|\log \Im(\sbi z)|^{m+n} +|\log \Im(\sbi \g z)|^{m+n}\right) \Im(\sbi \g z)^{\Re(s)}.
\end{equation}
The elementary inequality
\begin{equation*}
  |\log y|^k < (2/\varepsilon)^k(y^{\varepsilon k}+y^{-\varepsilon k})
\end{equation*}
for all $y,$ $\varepsilon >0$ and $k\gqs 0$, implies the first factor in \er{hyu} is
\begin{equation*}
 \ll  \Im(\sbi z)^{\varepsilon(m+n)}+ \Im(\sbi z)^{-\varepsilon(m+n)}+\Im(\sbi \g z)^{\varepsilon(m+n)}+\Im(\sbi \g z)^{-\varepsilon(m+n)}.
\end{equation*}
Hence
\begin{multline*}
  \sum_{\g \in \G_\cb\backslash \G} \left|C^a_{\g a}(f_{1}, \dots ,f_{m}) \overline{C^a_{\g a}(g_{1}, \dots ,g_{n})} \Im(\sbi \g z)^s \right| \\
  \ll E_\cb(z,\Re(s))+ E_\cb(z,\Re(s)+\varepsilon(m+n)) + E_\cb(z,\Re(s)-\varepsilon(m+n))
\end{multline*}
for any $\varepsilon >0$, with the implied constant depending on $\varepsilon$ and $z$. The result follows.
\end{proof}

\subsection{Transformation properties}  \label{32}

\begin{prop}\label{prop:tp}
For $\Re(s)>1$ and all $\delta \in \G$
\begin{equation} \label{eb}
  \mathcal{E}_\cb(\delta z,s) = I_{a}^{\delta a} \cdot \overline{J_{a}^{\delta a}} \cdot \mathcal{E}_\cb(z,s).
\end{equation}
\end{prop}
\begin{proof}
We have
\begin{align*}
  \mathcal{E}_\cb(\delta z,s) & = \sum_{\g \in \G_\cb\backslash \G}  I_{\g a}^{a} \cdot \overline{J_{\g a}^{a}} \cdot \Im(\sbi\g \delta z)^s \\
   & = \sum_{\g \in \G_\cb\backslash \G}  I_{\g \delta^{-1} a}^{a} \cdot \overline{J_{\g \delta^{-1} a}^{a}} \cdot \Im(\sbi\g z)^s.
\end{align*}
Then
 $ I_{\g \delta^{-1} a}^{a} = I_{\g \delta^{-1} a}^{\g a}I_{\g  a}^{a} = I_{ \delta^{-1} a}^{ a}I_{\g  a}^{a} = I^{ \delta a}_{ a}I_{\g  a}^{a}$
 and similarly  $J_{\g \delta^{-1} a}^{a}=J^{ \delta a}_{ a}J_{\g  a}^{a}.$
Using these relations we obtain
\begin{align*}
  \mathcal{E}_\cb(\delta z,s) & = \sum_{\g \in \G_\cb\backslash \G}  I^{ \delta a}_{ a}I_{\g  a}^{a} \cdot \overline{J^{ \delta a}_{ a}J_{\g  a}^{a}} \cdot \Im(\sbi\g z)^s\\
 & =  I^{ \delta a}_{ a} \overline{J^{ \delta a}_{ a}} \sum_{\g \in \G_\cb\backslash \G} I_{\g  a}^{a} \cdot \overline{J_{\g  a}^{a}} \cdot \Im(\sbi\g z)^s. \qedhere
\end{align*}
\end{proof}

The relation \er{eb} elegantly encapsulates a lot of information.
Comparing corresponding coefficients shows
\begin{multline}  \label{eb2}
  E_\cb(\delta z,s;f_1, \dots ,f_m, \overline{g_1}, \dots, \overline{g_n}) \\
  = \sum_{i=0}^m \sum_{j=0}^n
   C_a^{\delta  a}(f_1,  \dots,  f_i)\overline{C_a^{\delta  a}(g_1,  \dots,  g_j)}
   E_\cb(z,s;f_{i+1}, \dots ,f_m, \overline{g_{j+1}}, \dots, \overline{g_n}).
\end{multline}
When $\mathbf{f}=()$, a similar formula is given by Diamantis and Sim
\cite[Eq. (3.30)]{DS}.

\begin{example}{\rm Another special case of Proposition \ref{prop:tp}
    appears in \cite{JOS08}.  Suppose that $\mathcal{E}_{\cb}(z,s)$ is defined with modular symbols $I_a^b(f)$ and $J_a^b(g)$ with each depending on single cusp forms. Then the coefficients in the series \er{wo2} take the form
\begin{align}
  E_\cb(z,s;\underbrace{f, \dots ,f}_m, \underbrace{\overline{g}, \dots, \overline{g}}_n) & =  \sum_{\g \in \G_\cb\backslash \G} C^a_{\g a}(f, \dots ,f) \overline{C^a_{\g a}(g, \dots ,g)} \Im(\sbi \g z)^s \notag\\
   & =  \sum_{\g \in \G_\cb\backslash \G} \frac{1}{m!}\biggl( \int^a_{\g a} f(u) \, du \biggr)^m \frac{1}{n!}\biggl(\overline{ \int^a_{\g a} g(u) \, du} \biggr)^n \Im(\sbi \g z)^s \notag\\
   & =  \frac{(-1)^m}{(2\pi i)^{m+n}}  \frac{1}{m!n!} \sum_{\g \in \G_\cb\backslash \G} \s{\g}{f}^m \overline{ \s{\g}{g}}^n \Im(\sbi \g z)^s \label{emnx}
\end{align}
where the modular symbol pairing  was defined in \er{mod}.
The series in \er{emnx} is studied in detail in \cite{PeRi}, \cite{JOS08} and written as
\begin{equation}\label{emnw}
  E^{m,n}_\cb(z,s;f,g) := \sum_{\g \in \G_\cb\backslash \G} \s{\g}{f}^m \overline{ \s{\g}{g}}^n \Im(\sbi \g z)^s.
\end{equation}
It follows from relation \er{eb}, by translating the terms of \er{eb2} with \er{emnx}, that
\begin{equation*}
  E^{m,n}_\cb(\delta z,s;f,g) = \sum_{i=0}^m\sum_{j=0}^n \binom{m}{i} \binom{n}{j} (-\s{\g}{f})^i (-\overline{\s{\g}{g}})^j E^{m-i,n-j}_\cb(z,s;f,g)
\end{equation*}
for all $\delta \in \G.$ This is Lemma 4.1 of \cite{JOS08}.}
\end{example}

The results in Sections \ref{31} and \ref{32} complete the proof of Theorem \ref{thm:intro}, except for the verification of \er{lap}. But this follows from $\Delta y^s=s(s-1)y^s$ as in the classical case for $E_\cb(z,s)$.

\subsection{Fourier expansions}
Write $z=x+iy$ for $z\in \H$. Let $\ca$ and $\cb$ be cusps for $\G$. By \cite[Thm. 3.4]{IwSp}, the Fourier expansion of the classical real-analytic Eisenstein series is
\begin{equation}\label{rae}
  E_\ca(\sb z,s) = \delta_{\ca\cb}y^s +\phi_{\ca\cb}(s) y^{1-s} + \sum_{k \in \Z_{\neq 0}} \phi_{\ca\cb}(k,s) W_s(kz).
\end{equation}
The function $W_s(z)$ is a Whittaker function and $\phi_{\ca\cb}(s)$, $\phi_{\ca\cb}(k,s)$ may be expressed in terms of Kloosterman sums. Also $ \delta_{\ca\cb}$ takes the value $1$ if $\ca$ and $\cb$ are $\G$-equivalent and is $0$ otherwise.

The same proof gives the expansion
\begin{equation}\label{rae2}
  E_\ca(\sb z,s;f_1, \dots ,f_m, \overline{g_1}, \dots, \overline{g_n})) =  \phi_{\ca\cb}(s;f_1, \dots, \overline{g_n})) y^{1-s} + \sum_{k \in \Z_{\neq 0}} \phi_{\ca\cb}(k,s;f_1,  \dots, \overline{g_n})) W_s(kz)
\end{equation}
 for $m+n\gqs 1$.
The coefficients $ \phi_{\ca\cb}(s;f_1, \dots, \overline{g_n}))$ and $ \phi_{\ca\cb}(k,s;f_1, \dots, \overline{g_n}))$ may be written  in terms of  Kloosterman sums twisted by iterated integrals, similarly to \cite[Eqs. (1.2),(1.3)]{OS}. It is an interesting question to find explicit forms for these coefficients. See \cite[Remark 5.5]{GOS} and \cite[Sect. 4.1]{Ri} for the determination of $\phi_{\ca\cb}(s;f_1)$. In Bruggeman-Diamantis  \cite{BD},   $\phi_{\ca\cb}(s;f_1)$ and $\phi_{\ca\cb}(k,s;f_1)$ are expressed in terms of $L$-functions and shifted convolution sums.

Combining the expansions \er{rae2} gives the Fourier expansion of $\mathcal{E}_{\ca}(z,s)$ as
\begin{equation}\label{rae3}
  \mathcal{E}_\ca(\sb z,s) = \delta_{\ca\cb}y^s + \Phi_{\ca\cb}(s) y^{1-s} + \sum_{k \in \Z_{\neq 0}} \Phi_{\ca\cb}(k,s) W_s(kz)
\end{equation}
where
\begin{multline} \label{phab}
 \Phi_{\ca\cb}(s)  =\phi_{\ca\cb}(s) +  \sum_i  \phi_{\ca\cb}(s; f_i) \cdot X_i +  \sum_i  \phi_{\ca\cb}(s; \overline{g_i}) \cdot \overline{Y_i}\\
  + \sum_{i,j} \phi_{\ca\cb}(s; f_i, f_j) \cdot X_i X_j +\sum_{i,j} \phi_{\ca\cb}(s; f_i, \overline{g_j}) \cdot X_i \overline{Y_j}
  +\sum_{i,j}\phi_{\ca\cb}(s; \overline{g_i}, \overline{g_j})\cdot \overline{Y_i Y_j} + \cdots
\end{multline}
and  $\Phi_{\ca\cb}(k,s)$ is a similar series.

\section{Meromorphic continuation}

The meromorphic continuation of the classical Eisenstein series
$E_\ca(z,s)$ to all $s\in \C$ may be shown using its Fourier
expansion in simple cases, such as when $\G=\SL_2(\Z)$. Selberg proved
the continuation for general groups $\G$ and one of Selberg's methods, as described in \cite[Chap. 6]{IwSp},
was extended in
\cite{JOS08} to prove the continuation of $E^{m,n}_\ca(z,s;f,g)$. We recall from \er{emnx} that this is a constant times
\begin{equation}\label{eee}
  E_\ca(z,s;f_{1}, \dots
  ,f_{m},\overline{g_{1}},\dots,\overline{g_{n}})
\end{equation}
where $f_1=f_2=\cdots =f_m=f$ and $g_1=g_2=\cdots=g_n=g$.  The proof
in \cite{JOS08} goes through almost without change in the general case
of distinct functions $f_i$ and $g_j$. The two main properties
that the proof needs are that \er{eee} is an eigenfunction of the Laplacian
and that it transforms into itself as $z$ is replaced by $\g z$ for
$\g \in \G$ except for the addition of a lower-order term. The proof
also requires some growth estimates that we describe next.

\subsection{Growth estimates for twisted Eisenstein series} \label{jhb}
  Following \cite{IwSp}, we may measure the
growth of $\G$-invariant functions in terms of the  {\it invariant height} function
defined by
$$
y_\G(z):=\max_\ca \max_{\g \in \G}(\Im( \sa^{-1}\g z))
$$
for $z\in \H$ where the outer maximum is taken over our fixed set of inequivalent cusps.
Thus $y_\G(z)$ approaches $\infty$ as  $z$ approaches any cusp.

The twisted Eisenstein series are not
 $\G$-invariant and it is more convenient to fix a
fundamental domain $\F$ and examine their growth there. Let $\mathcal P_Y \subset \H$ denote the strip with $|x| \leqslant 1/2$ and $y \geqslant Y$. We choose $\F$ so that its closure contains the cuspidal zones $\sa \mathcal P_Y$ for all $\ca$ and $Y$ large enough; see \cite[Section 2.2]{IwSp}. For $z \in \F$ we define the  {\it domain height} function
$$
y_\F(z):=\max_\ca (\Im( \sa^{-1} z)).
$$
Clearly, $y_\F(z)$ is bounded below by a positive constant for $z \in \F$ and bounded above by
$
y_\G(z)
$.
In fact $y_\F(z)=y_\G(z)$ when $z$ is in the cuspidal zones $\sa \mathcal P_Y$ for all $\ca$ and $Y$ large enough. Hence $y_\G(z) \ll y_\F(z)$ for $z \in \F$.
It is also shown in \cite[Lemma A.1]{JOS08} that, for any cusp $\cb$ and all $z\in \H$
\begin{equation}\label{ygz}
  y_\G(\sb z) \lqs (c_\G+ 1/c_\G)(y+1/y)
\end{equation}
for $c_\G$ depending only on $\G$.

With Proposition \ref{poi} we have shown that, for any cusp $\cb$,
\begin{equation}\label{ty}
   C_a^{z}(f_1, f_2, \dots,  f_n) \ll 1+|\log \Im(\sbi z) |^n
\end{equation}
for all $z \in \H$. It follows from \er{ty} that
\begin{equation}\label{ty2}
   C_a^{z}(f_1, f_2, \dots,  f_n) \ll \log^n(y_\F(z)+e)
\end{equation}
for all $z \in \F$.
We have, writing $\sigma=\Re(s)$ as usual,
\begin{equation}\label{bd2}
E_\ca(z,s) \ll y_\G(z)^\sigma
\end{equation}
as in \cite[Corollary 3.5]{IwSp} for example. Using \er{ty2} in the proof of \cite[Prop. 3.3]{JOS08} shows
\begin{equation}\label{eee2}
  E_\ca(z,s;f_{1}, \dots
  ,f_{m},\overline{g_{1}},\dots,\overline{g_{n}}) \ll y_\F(z)^{1-\sigma+\varepsilon}
\end{equation}
for all $\varepsilon>0$ and all $z \in \F$ when $m+n>0$.

Our next goal is to extend \er{eee2} to a bound on all of $\H$. Similarly to \cite[Sect. 5.1]{JOS08} we may do this neatly using the series
\begin{equation}
  Q_\cb(z,s;f_{1}, \dots ,f_{m},\overline{g_{1}},\dots,\overline{g_{n}})  := \sum_{\g \in \G_\cb\backslash \G} C^a_{\g z}(f_{1}, \dots ,f_{m}) \cdot  \overline{C^a_{\g z}(g_{1}, \dots ,g_{n})} \cdot  \Im(\sbi \g z)^s. \label{qeis2}
\end{equation}

\begin{prop}
The series $Q_\cb(z,s;f_{1}, \dots ,f_{m},\overline{g_{1}},\dots,\overline{g_{n}})$ converges absolutely for $\Re(s)>1$ to a $\G$-invariant function of $z$. It satisfies the bound
\begin{equation}\label{qmn2}
Q_\cb(z,s;f_{1}, \dots ,f_{m},\overline{g_{1}},\dots,\overline{g_{n}}) \ll \log^{m+n}(y_\G(z)+e) \cdot  y_\G(z)^\sigma
\end{equation}
for all $z\in \H$ and is related to the twisted Eisenstein series through the identities
\begin{align}  \label{qtew}
   \begin{split}
   Q_\cb(z,s;f_1, \dots ,f_m, \ &\overline{g_1}, \dots, \overline{g_n}) \\
   &= \sum_{i=0}^m \sum_{j=0}^n
   C_z^{a}(f_1,  \dots,  f_i) \cdot \overline{C_z^{a}(g_1,  \dots,  g_j)}
    \cdot E_\cb(z,s;f_{i+1}, \dots ,f_m, \overline{g_{j+1}}, \dots, \overline{g_n}),
   \end{split}\\
\begin{split} \label{qtew2}
E_\cb(z,s;f_1, \dots ,f_m, \ &\overline{g_1}, \dots, \overline{g_n}) \\
  &= \sum_{i=0}^m \sum_{j=0}^n
   C_a^{z}(f_1,  \dots,  f_i) \cdot \overline{C_a^{z}(g_1,  \dots,  g_j)}
    \cdot Q_\cb(z,s;f_{i+1}, \dots ,f_m, \overline{g_{j+1}}, \dots, \overline{g_n}).
    \end{split}
\end{align}
\end{prop}
\begin{proof}
We first see that \er{qtew} follows from the identities $I_{\g z}^a=I_{\g z}^{\g a} I_{\g a}^a = I_{ z}^{ a} I_{\g a}^a$. Hence \er{qeis2} converges absolutely for $\Re(s)>1$. It is $\G$-invariant because replacing $z$ by $\delta z$ for $\delta \in \G$ just reorders the series.

Use \er{ty2}, \er{bd2} and \er{eee2} to bound the right side of \er{qtew} and obtain
\begin{equation}\label{qmn2f}
Q_\cb(z,s;f_{1}, \dots ,f_{m},\overline{g_{1}},\dots,\overline{g_{n}}) \ll \log^{m+n}(y_\F(z)+e) \cdot  y_\F(z)^\sigma
\end{equation}
for all $z\in \F$. Then \er{qmn2} is a consequence of \er{qmn2f} since $Q_\cb$ is $\G$-invariant.

The relation \er{qtew2} follows similarly to  \er{qtew}.
\end{proof}

\begin{cor} \label{kol}
For every cusp $\cb$ and all $z\in \H$ we have
\begin{equation}\label{ghjx}
  E_\ca(\sb z,s;f_1, \dots ,f_m, \overline{g_1}, \dots, \overline{g_n}) \ll \log^{m+n}(y+1/y) \cdot (y^\sigma + y^{-\sigma}).
\end{equation}
\end{cor}
\begin{proof}
Use \er{ygz} in \er{qmn2} to show that
\begin{equation}\label{azd}
  Q_\ca(\sb z,s;f_{1}, \dots ,f_{m},\overline{g_{1}},\dots,\overline{g_{n}}) \ll \log^{m+n}(y+1/y) \cdot (y^\sigma + y^{-\sigma})
\end{equation}
for all $z\in \H$. Then \er{azd} and \er{ty}
 along with the inequality $|\log y|+1 < 3\log(y+1/y)$  bound the right side of \er{qtew2}.
\end{proof}

The estimates we have developed in this subsection are the same as those found in Sections 3.2 and 5.1 of \cite{JOS08} for the case of equal cusp forms: $f_1=\cdots =f_m$ and  $g_1=\cdots=g_n$.

\subsection{Continuation using the Fredholm theory of integral equations} \label{contf}
We are following Sections 5.2 and 5.3 of \cite{JOS08} in the next discussion. Set
\begin{equation*}
  u(z,w):=\frac{|z-w|^2}{4 \Im (z) \Im (w)}, \qquad G_\alpha(u):=\frac{1}{4\pi}\int_0^1(t(1-t))^{\alpha-1}(t+u)^{-\alpha}\,dt.
\end{equation*}
It is convenient to let $\mathbf{f}=(f_1, \dots ,f_m)$ and $\mathbf{g}=(g_1, \dots ,g_n)$.
The resolvent of the Laplacian may be written as an integral operator with kernel given by the above Green function $G_\alpha(u)$. Since $E_\ca(z,s;\mathbf{f}, \overline{\mathbf{g}})$ is an eigenfunction of $\Delta$ and Corollary \ref{kol} holds, we obtain
\begin{equation}\label{rev}
  \frac{-E_\ca(z,s;\mathbf{f}, \overline{\mathbf{g}})}{(\alpha(1-\alpha)-s(1-s))}=\int_{\H}G_\alpha(u(z,z'))E_\ca(z',s;\mathbf{f}, \overline{\mathbf{g}})\,d\mu (z')
\end{equation}
for $1<$ Re$(s)<\alpha-2$. To remove a logarithmic singularity at $u=0$ we set $G_{\alpha \beta}(u):=G_\alpha(u)-G_\beta(u)$ for $\beta<\alpha$. Also put
\begin{equation*}
  \nu_{\alpha \beta}(s):=(\alpha(1-\alpha)-s(1-s))^{-1}-(\beta(1-\beta)-s(1-s))^{-1}
\end{equation*}
and, on breaking up $\H$ into the images of $\F$ under $\G$, \er{rev} becomes
\begin{equation}\label{rev2}
  -\nu_{\alpha \beta}(s) E_\ca(z,s;\mathbf{f}, \overline{\mathbf{g}})=\int_{\F} \sum_{\g \in \G} G_{\alpha\beta}(u(z,\g z'))E_\ca(\g z',s;\mathbf{f}, \overline{\mathbf{g}})\,d\mu (z')
\end{equation}
for $1<$ Re$(s)<\beta-2<\alpha-2$. Use the transformation property \er{eb2} of $E_\ca(\g z',s;\mathbf{f}, \overline{\mathbf{g}})$ to expand this. With
\begin{equation*}
  G_{\alpha \beta}^{i,j}(z,z'):=\sum_{\g \in \G}  C_a^{\g a}(f_1,  \dots,  f_i)\overline{C_a^{\g a}(g_1,  \dots,  g_j)} G_{\alpha\beta}(u(z,\g z'))
\end{equation*}
and $ G_{\alpha \beta}(z,z'):= G_{\alpha \beta}^{0,0}(z,z')$ we find
\begin{equation}\label{fredx}
   E_\ca(z,s;\mathbf{f}, \overline{\mathbf{g}})=q^{m,n}(z,s)+\lambda \int_{\F}  G_{\alpha\beta}(z, z')E_\ca(\g z',s;\mathbf{f}, \overline{\mathbf{g}})\,d\mu (z')
\end{equation}
for $\lambda=\lambda(s):=-1/\nu_{\alpha \beta}(s)$, a polynomial in $s$ of degree $4$, and
\begin{equation} \label{fredx2}
  q^{m,n}(z,s) := \lambda \sum_{i,j} \int_{\F}
   G_{\alpha \beta}^{i,j}(z,z')
   E_\ca(z',s;f_{i+1}, \dots ,f_m, \overline{g_{j+1}}, \dots, \overline{g_n})\,d\mu (z')
\end{equation}
where the sum in \er{fredx2} is over all $i,$ $j$ that  satisfy $0\lqs i\lqs m$, $0\lqs j\lqs n$ and are not both zero.

The integral equation \er{fredx} is valid for $1<$ Re$(s)<\beta-2$. However, it determines $E_\ca(z,s;\mathbf{f}, \overline{\mathbf{g}})$ uniquely and by the Fredholm theory there exists a kernel $D_\lambda(z,z')$ and a function $D(\lambda) \not\equiv 0$, which are both built from $G_{\alpha \beta}(z,z')$ and analytic in $\lambda$, so that
\begin{equation}\label{drf}
  E_\ca(z,s;\mathbf{f}, \overline{\mathbf{g}}) = q^{m,n}(z,s) +\frac{\lambda}{D(\lambda)} \int_{\F}  D_\lambda(z,z') q^{m,n}(z',s) \,d\mu (z').
\end{equation}
This provides the desired analytic continuation of $E_\ca(z,s;\mathbf{f}, \overline{\mathbf{g}})$ with an induction argument when we know that the lower-order series $q^{m,n}(z,s)$ already has a continuation.

We simplified the  presentation above by omitting a step. The kernel $G_{\alpha\beta}(z, z')$ in \er{fredx} is not bounded, as required by  the Fredholm theory we are using, and must be replaced by a truncated version $\eta(z)\eta(z')H_s(z,z')$ as described in \cite[p. 84]{IwSp} and \cite[Sect. 5.3]{JOS08}.

Theorem \ref{thm:con} follows immediately from the next result.

\begin{theorem} \label{cont2}
Let $\mathbf{f}=(f_1, \dots ,f_m)$ and $\mathbf{g}=(g_1, \dots ,g_n)$ for $m,n \in \Z_{\gqs 0}$.
For every ball $\B_r \subset \C$ of radius $r$ about the origin  there exist functions $A_\ca(s)$, $\phi_{\ca \cb}(s;\mathbf{f}, \overline{\mathbf{g}})$ and $\phi_{\ca \cb}(k,s;\mathbf{f}, \overline{\mathbf{g}})$ for all   $k \in \Z_{\neq 0}$ so that the following assertions hold:
\begin{enumerate}
\item $A_\ca(s)$ is analytic on $\B_r$ and not identically $0$.
\item $\phi_{\ca \cb}(s;\mathbf{f}, \overline{\mathbf{g}})$ and $\phi_{\ca \cb}(k,s;\mathbf{f}, \overline{\mathbf{g}})$ are meromorphic functions of $s$ on $\B_r$.
\item For all $s\in \B_r$
\begin{align*}
    A^{m+n+1}_\ca(s) \cdot \phi_{\ca \cb}(s;\mathbf{f}, \overline{\mathbf{g}}) & \ll  1,\\
    A^{m+n+1}_\ca(s) \cdot \phi_{\ca \cb}(k,s;\mathbf{f}, \overline{\mathbf{g}}) & \ll  (\log^{m+n} |k| +1)(|k|^\sigma + |k|^{1-\sigma}).
\end{align*}
\item The Fourier expansion
\begin{equation} \label{fexpmn}
E_{\ca}(\sb z,s;\mathbf{f}, \overline{\mathbf{g}})=\delta^{m,n}_{0,0} \cdot \delta_{\ca \cb}
y^s+ \phi_{\ca \cb}(s;\mathbf{f}, \overline{\mathbf{g}})y^{1-s}+ \sum_{k\not=0}\phi_{\ca
\cb}(k,s;\mathbf{f}, \overline{\mathbf{g}})W_s(kz)
\end{equation}
 agrees with \er{rae}, \er{rae2} for $\Re(s)>1$ and, for all $z \in \H$, converges to a meromorphic function of  $s \in\B_r$.
 \item For all $s \in\B_r$ and $z\in \F$ we have
\begin{equation} \label{fbnd}
A^{m+n+1}_\ca(s)  \cdot E_{\ca}(z,s;\mathbf{f}, \overline{\mathbf{g}}) \ll  y_\F(z)^{|\sigma-1/2|+1/2}.
\end{equation}
\end{enumerate}
\end{theorem}

\begin{proof}
The case $m+n=0$ is given  in  \cite[Thm. 6.1]{JOS08}. Note that it should have been stated there that $A_\ca(s)$ is  not identically $0$. This follows from its construction from $D(\lambda)$ in \cite[Eq. (6.6)]{JOS08} and the fact that $D(0)=1$ from \cite[p. 193]{IwSp}.

The rest of the proof uses induction on $m+n$. This is the same as the proof of \cite[Thm. 6.5]{JOS08}, but based on the integral equation \er{fredx} and using the bounds from Section \ref{jhb}.
\end{proof}

The techniques of Diamantis-Sim \cite{DS} should also give the meromorphic continuation of $E_{\ca}(z,s;\mathbf{f}, \overline{\mathbf{g}})$ to all $s\in \C$. There, they essentially find the continuation of $Q_{\ca}(z,s;\mathbf{f}, \overline{\mathbf{g}})$, as defined in \er{qeis2}, by means of its spectral expansion. Their proof is for $\mathbf{f}$ empty but should carry over to our
  setting without difficulty; see \cite[Theorem 3.4]{DS}.

The first proof of the continuation of $E^{m,n}_\ca(z,s;f,g)$, corresponding to the case of equal cusp forms $f_1=\cdots =f_m$ and  $g_1=\cdots=g_n$, was given in Petridis-Risager \cite{PeRi}, following earlier work in Petridis \cite{Pe}. This method exploits the fact that
\begin{equation}\label{chi}
  \chi^f_\varepsilon(\g)=\chi_\varepsilon(\g):= \exp\left( \varepsilon \int_{z_0}^{\g z_0} f(u)\, du  \right)
\end{equation}
is a  character with $\chi_\varepsilon(\g_1\g_2)=\chi_\varepsilon(\g_1)\chi_\varepsilon(\g_2)$. Setting
\begin{equation} \label{wjk}
  E_{\cb}(z,s; \chi_\varepsilon):=\sum_{\g
    \in \G_\cb\backslash \G} \chi_\varepsilon(\g) \cdot \Im(\sbi\g z)^s,
\end{equation}
we find that $E_{\cb}(z,s; f)$ can be recovered as $\frac{d}{d\varepsilon} E_{\cb}(z,s; \chi_\varepsilon)|_{\varepsilon=0}$. To obtain $E^{m,n}_\ca(z,s;f,g)$, the  character $\chi_\varepsilon$ is replaced with a product  $\chi_{\varepsilon_1} \cdot \chi_{\varepsilon_2}\cdots$ with different parameters; see \cite[Eq. (1.10)]{PeRi}. The meromorphic continuation of \er{wjk} and this many parameter generalization is achieved in \cite{Pe,PeRi} by employing spectral deformation theory. Taking derivatives then gives the continuation of $E^{m,n}_\ca(z,s;f,g)$.

It is not clear if the methods of \cite{Pe,PeRi} extend to proving the
continuation of the general Eisenstein series
$E_{\ca}(z,s;\mathbf{f}, \overline{\mathbf{g}})$ studied in this
paper. For example, if the integral in \er{chi} is replaced by
$C_{z_0}^{\g z_0}(f_1,f_2)$ then $ \chi_\varepsilon(\g)$ will not be a
character in general. By Example \ref{exx} (iv) and \er{hm3}, it will
satisfy the more complicated relation
\begin{equation*}
  \chi_\varepsilon(\g_1\g_2\g_3)=\frac{\chi_\varepsilon(\g_1\g_2)\chi_\varepsilon(\g_1\g_3)\chi_\varepsilon(\g_2\g_3)}
  {\chi_\varepsilon(\g_1)\chi_\varepsilon(\g_2)\chi_\varepsilon(\g_3)}.
\end{equation*}

\subsection{Functional equations}

By analogy with the classical Eisenstein series, we expect a relationship between the values of $E_{\ca}( z,s;\mathbf{f}, \overline{\mathbf{g}})$ at $s$ and $1-s$. In some cases we do have such a functional equation and to express it we  set up the following notation.

Suppose $\G$ has $h$ inequivalent cusps. Let $E(z,s;\mathbf{f}, \overline{\mathbf{g}})$ be the $h\times 1$ column vector with entries $E_\ca(z,s;\mathbf{f}, \overline{\mathbf{g}})$ as $\ca$ lists the inequivalent cusps. With the same ordering, let $\phi(s;\mathbf{f}, \overline{\mathbf{g}})$ be the $h\times h$ matrix with entries $\phi_{\ca\cb}(s;\mathbf{f}, \overline{\mathbf{g}})$.

The formal series version of $E$ above is  $\mathcal{E}(z,s)$, the $h\times 1$ column vector with entries $\mathcal{E}_\ca(z,s)$. It satisfies the vector version of \er{wo2}:
\begin{equation*}
  \mathcal{E}(z,s) = \sum_{c,d \gqs 0} \sum_{i_1, \dots, i_c,j_1,\dots, j_d} E(z,s;f_{i_1}, \dots, f_{i_c}, \overline{g_{j_1}}, \dots, \overline{g_{j_d}})
  \cdot X_{i_1} \cdots  X_{i_c} \overline{Y_{j_1}} \cdots \overline{Y_{j_d}} .
\end{equation*}
 The formal series version of $\phi$ is $\Phi(s)$, the $h\times h$ matrix with entries $\Phi_{\ca\cb}(s)$. It satisfies the matrix version of \er{phab}:
\begin{equation*}
  \Phi(s) = \sum_{c,d \gqs 0} \sum_{i_1, \dots, i_c,j_1,\dots, j_d} \phi(s;f_{i_1}, \dots, f_{i_c}, \overline{g_{j_1}}, \dots, \overline{g_{j_d}})
  \cdot X_{i_1} \cdots  X_{i_c} \overline{Y_{j_1}} \cdots \overline{Y_{j_d}} .
\end{equation*}

We may give a simple reformulation of  \cite[Thm. 7.1]{JOS08} as follows.

\begin{theorem}\label{thm:fe}
Define $\mathcal{E}( z,s)$ with modular symbols $I_{\g a}^a(f)$ and $J_{\g a}^a(g)$ depending on single cusp forms. It satisfies the functional equation
\begin{equation}\label{fe1}
  \Phi(1-s)  \mathcal{E}( z,s) = \mathcal{E}( z,1-s)
\end{equation}
with
\begin{equation}\label{fe11}
  \Phi(1-s) \Phi(s)=I
\end{equation}
for $I$  the $h\times h$ identity matrix.
\end{theorem}

The formal series equations \er{fe1} and \er{fe11} are equivalent to showing the matrix equations
\begin{align}\label{tex}
 \sum_{i=0}^m \sum_{j=0}^n \phi(1-s;\overbrace{f, \dots,f}^i,\overbrace{\overline{g},\dots,\overline{g}}^j) E(z,s; \overbrace{f, \dots,f}^{m-i},\overbrace{\overline{g},\dots,\overline{g}}^{n-j})  & = E(z,1-s;\overbrace{f, \dots,f}^m,\overbrace{\overline{g},\dots,\overline{g}}^n), \\
  \sum_{i=0}^m \sum_{j=0}^n \phi(1-s;\overbrace{f, \dots,f}^i,\overbrace{\overline{g},\dots,\overline{g}}^j) \phi(s; \overbrace{f, \dots,f}^{m-i},\overbrace{\overline{g},\dots,\overline{g}}^{n-j}) & =
  \left\{
    \begin{array}{ll}
      I, & \hbox{if $m=n=0$;} \\
      0, & \hbox{otherwise}
    \end{array}
  \right. \label{tex2}
\end{align}
for all $m,$ $n \in \Z_{\gqs 0}$. The $m=n=0$ cases of \er{tex} and \er{tex2} are the classical functional equations
\begin{equation*}
  \phi(1-s) E(z,s)   = E(z,1-s), \qquad \phi(1-s)\phi(s)   = I
\end{equation*}
as shown in \cite[Sect. 6.3]{IwSp}.

In fact the functional equation of Theorem \ref{thm:fe} appears
already in Petridis \cite{Pe}. We may write the functional equation
\cite[Eq. 1.9]{Pe} of the series \er{wjk} as the matrix equation
\begin{equation} \label{wjk2}
  \phi(1-s;\chi^f_\varepsilon \cdot \chi^{\overline{g}}_{\varepsilon'}) E(z,s;\chi^f_\varepsilon \cdot \chi^{\overline{g}}_{\varepsilon'})   = E(z,1-s;\chi^f_\varepsilon \cdot \chi^{\overline{g}}_{\varepsilon'}).
\end{equation}
Now we see that \er{wjk2}, treated as a relation of formal series in
$\varepsilon$ and $\varepsilon'$, agrees with \er{fe1}.

See also \cite[Thm. 43]{Ri} where a functional equation equivalent to \er{fe1} is proved, but with scattering matrix $\Phi(1-s)$ defined differently.

It is  natural to ask whether $ \mathcal{E}( z,s)$ satisfies the
functional equation \er{fe1}, or some other one, when it contains modular symbols $I_{\g a}^a(\mathbf{f})$ and $J_{\g a}^a(\mathbf{g})$ depending on more than one cusp form.

\section{Higher-order automorphic forms and maps}

We see next how the series $E_\cb(z,s;\mathbf{f}, \overline{\mathbf{g}})$ and the iterated integrals $C_a^{z}(\mathbf{f})$ and $C_a^{\g a}(\mathbf{f})$ fit into a  larger framework.

\subsection{Higher-order forms}

Following the description in \cite[Sect. 3]{JOS08}, we may define a sequence  $\mathcal A^n(\G)$ of sets of smooth functions from $\H \to \C$ recursively as follows. Let $\mathcal A^0(\G):=\{\H \to 0\}$ and  for  $n \in \Z_{\geqslant 1}$ set
$$
\mathcal A^n(\G):=\Bigl\{\psi \, \Big| \, \psi(\g z) -\psi(z) \in \mathcal A^{n-1}(\G) \text{ \ for all \ } \g \in \G \Bigr\}.
$$
Elements of $\mathcal A^n(\G)$ are called {\em $n$th-order automorphic forms}, naturally forming a vector space over $\C$. The classical $\G$-invariant functions, such as $E_\ca(z,s)$, are in
$\mathcal A^1(\G)$ and so are first-order.
If we let $\g \in \G$ act on $\psi$ by $(\psi|\g)(z):=\psi(\g z)$, and extend this action to all $\C[\G]$ by linearity, then we see that $\psi \in \mathcal A^{n}(\G)$ if and only if
\begin{equation*}
    \psi \big|(\g_1 -I)(\g_2-I) \cdots (\g_n-I) = 0  \quad \text{for all} \quad \g_1,\g_2, \dots,\g_n \in \G.
\end{equation*}
Inductive arguments show that
\begin{equation} \label{pmna}
  \mathcal A^m(\G) \subseteq \mathcal A^n(\G) \quad \text{for all} \quad 0 \leqslant m \leqslant n
\end{equation}
and if $\phi(z) \in \mathcal A^m(\G)$ and  $\psi(z) \in \mathcal A^n(\G)$ then
\begin{equation*}
  \phi(z)\cdot \psi(z) \in \mathcal A^{m+n-1}(\G)
\end{equation*}
for $m+n \geqslant 1$.
With \eqref{eb2}, $E_\ca(\g z,s;f) - E_\ca(z,s;f) \in \mathcal A^1(\G)$ implying that $E_\ca(z,s;f) \in \mathcal A^2(\G)$. So $E_\ca(z,s;f)$ is a second-order form satisfying \er{eq:1}.

\begin{prop} \label{ccee}
For all $m,$ $n \in \Z_{\geqslant 0}$ we have
\begin{enumerate}
  \item $C_a^z(f_1, \dots ,f_m) \in \mathcal A^{m+1}(\G)$,
  \item $E_\ca(z,s;f_1, \dots ,f_m, \overline{g_1}, \dots, \overline{g_n}) \in \mathcal A^{m+n+1}(\G)$.
\end{enumerate}
\end{prop}
\begin{proof}
Part (i) is proved by induction on $m$. The $m=0$ case is true since $C_a^{z}()=1$, so assume $m\gqs 1$. With
\begin{equation*}
  I_a^{\g  z} = I_a^{\g  a} I_{\g a}^{\g z}=I_a^{\g  a} I_{a}^{z}
\end{equation*}
we find
\begin{equation} \label{trk}
  C_a^{\g z}(f_1, f_2, \dots,  f_m) - C_a^{z}(f_1, f_2, \dots,  f_m)
  = \sum_{j=1}^{m} C_a^{\g a}(f_1,  \dots,  f_j) C_a^{z}(f_{j+1},  \dots,  f_m).
\end{equation}
By induction $C_a^{z}(f_{j+1},  \dots,  f_m) \in \mathcal A^{m-j+1}(\G)$. With \er{pmna}, the right side of \er{trk} is in $\mathcal A^{m}(\G)$ and hence $C_a^z(f_1, \dots ,f_m) \in \mathcal A^{m+1}(\G)$, completing the induction.

Part (ii) has a similar proof using \eqref{eb2}.
\end{proof}

Note that results equivalent to Proposition \ref{ccee} are proved in \cite[Sect. 3.2]{DS}.

\subsection{Higher-order maps}
As in \cite[Sect. 10]{IO}, one can define a related sequence $\operatorname{Hom}^{[n]}(\G,\C)$ of sets of
functions from $\G$ to $\C$ as follows. Let $\operatorname{Hom}^{[0]}(\G,\C):=\{\G \to 0\}$. With the notation
$L_\delta(\g):=L(\g \delta)-L(\g)$ and $n \in \Z_{\geqslant 1}$  define
$$
\operatorname{Hom}^{[n]}(\G,\C) :=\Bigl\{L:\G \to \C \, \Big| \, L_\delta \in \operatorname{Hom}^{[n-1]}(\G,\C)  \text{ \ for all \ }  \delta \in \G \Bigr\}.
$$
For $L:\G\to \C$ and $\g \in \G$,
set $L|\g:=L(\g)$ and extend this linearly to all $\C[\G]$. Then $L \in \operatorname{Hom}^{[n]}(\G,\C)$ if and only if
$$
 L\big|(\g_1 -I)(\g_2-I) \cdots (\g_n-I) = 0  \quad \text{for all} \quad  \g_1,\g_2, \dots,\g_n \in \G .
$$
We see that $\operatorname{Hom}^{[1]}(\G,\C)$ is the space of constant functions. Elements $L$ of $\operatorname{Hom}^{[2]}(\G,\C)$ satisfy
\begin{equation*}
    L(\g_1\g_2)-L(\g_1)-L(\g_2)+L(I) = 0 \quad \text{for all} \quad \g_1,\g_2 \in \G,
\end{equation*}
making $\g \mapsto L(\g)-L(I)$ a homomorphism into the additive part of $\C$.
Similarly, elements $L$ of $\operatorname{Hom}^{[3]}(\G,\C)$ satisfy
\begin{equation} \label{hm3}
    L(\g_1\g_2\g_3)-L(\g_1\g_2)-L(\g_1\g_3)-L(\g_2\g_3)+L(\g_1)+L(\g_2)+L(\g_3)-L(I) = 0
\end{equation}
for all $\g_1,\g_2 ,\g_3 \in \G$. We may call elements of $\operatorname{Hom}^{[n]}(\G,\C)$  {\em $n$th-order maps} from $\G$ to $\C$ and they form a complex vector space.
Inductive arguments demonstrate that
\begin{equation*}
  \operatorname{Hom}^{[m]}(\G,\C) \subseteq \operatorname{Hom}^{[n]}(\G,\C) \quad  \text{for all} \quad 0 \leqslant m \leqslant n
\end{equation*}
and if $L \in \operatorname{Hom}^{[m]}(\G,\C)$ and  $L' \in \operatorname{Hom}^{[n]}(\G,\C)$ then
\begin{equation} \label{prod}
  L \cdot L' \in \operatorname{Hom}^{[m+n-1]}(\G,\C)
\end{equation}
for $m+n \geqslant 1$. If $\g$ is an elliptic element of $\G$, with $\g^N=I$ for some $N>0$, then a similar proof to \cite[Lemma 17]{IO} shows $L(\g)=L(I)$ for all $L \in \operatorname{Hom}^{[n]}(\G,\C)$ with $n \gqs 0$.

\begin{example}\label{exx}{\rm
We have the following examples of higher-order maps:
\begin{enumerate}
  \item The  modular symbol map $\g \mapsto \s{\g}{f}$ is in $\operatorname{Hom}^{[2]}(\G,\C)$. In fact, if we define
\begin{equation*}
  \operatorname{Hom}_0^{[n]}(\G,\C) :=\Bigl\{L \in \operatorname{Hom}^{[n]}(\G,\C) \, \Big| \, L(\g)=0 \ \text{ \ for all parabolic \ $\g \in \G$}  \Bigr\}
\end{equation*}
then $\s{\g}{f}$ is in $\operatorname{Hom}_0^{[2]}(\G,\C)$. Note that $L\in \operatorname{Hom}_0^{[n]}(\G,\C)$ implies $L(I)=0$.
  \item It now follows from \eqref{prod} that, for example,
\begin{equation} \label{sprod}
  \s{\g}{f}^m \overline{ \s{\g}{f}}^n \in \operatorname{Hom}_0^{[m+n+1]}(\G,\C).
\end{equation}
\item If $\psi \in \mathcal A^{n}(\G)$ then, for fixed $z_0 \in \H$, $L(\g)$ defined as $\psi(\g z_0)$ is in $\operatorname{Hom}^{[n]}(\G,\C)$.
  \item As  functions from $\G$ to $\C$, the iterated integral $C_a^{\g b}(f_1, f_2, \dots,  f_n)$ and its complex conjugate are in $\operatorname{Hom}^{[n+1]}(\G,\C)$. This follows from the previous example and Proposition \ref{ccee} (i). With the identity \er{iabc}, we may allow $a$ and $b$ to be in $\H \cup \{\mathtt{cusps}\}$. By Corollary \ref{cr} these functions are in the subspace $\operatorname{Hom}_0^{[n+1]}(\G,\C)$ for $n\gqs 1$.
  \item An interesting  third-order map $\theta_\ca$ is obtained in the paper \cite{JOSS}. The Kronecker limit formula gives the first two terms in the Laurent expansion of $E_\ca(z,s)$ at $s=1$. From the second term we may derive the {\em modular Dedekind symbol} $S_\ca$ which is a map from $\G$ to $\R$.
      Doing the same with the Eisenstein series $E_\ca^{m,m}(z,s;f,f)$ from \er{emnw}
      for $m\geqslant 1$ produces the {\em higher-order  modular Dedekind symbol} $S_\ca^*$ which is independent of $m$ and is also  a map from $\G$ to $\R$. Then, as shown in \cite[Sect. 5.5]{JOSS}
      \begin{equation*}
        \theta_\ca:=S_\ca^*-S_\ca \in \operatorname{Hom}^{[3]}(\G,\C).
      \end{equation*}
      However $\theta_\ca$ is not zero on all parabolic elements
      and so cannot be expressed in terms of products such as \er{sprod} or the iterated integral $C_a^{\g b}(f_1, f_2)$. It would be interesting to understand how
these higher-order modular Dedekind symbols fit into the context of the
 noncommutative Dedekind symbols Manin introduces in \cite{M3}.
\end{enumerate}}
\end{example}

\subsection{The subspace $\operatorname{H}^{[n]}(\G,\C)$}
For fixed $a \in \H \cup \{\mathtt{cusps}\}$, we see by Example \ref{exx} (iv) and \er{prod} that
\begin{equation} \label{frm}
  C_a^{\g a}(f_1, \dots ,f_j) \overline{ C_a^{\g a}(f_{j+1}, \dots ,f_m)}
\end{equation}
is in  $\operatorname{Hom}_0^{[m+1]}(\G,\C)$ for $m\gqs 1$.
For $n\gqs 1$, let $\operatorname{H}^{[n]}(\G,\C)$ be the subspace of
$\operatorname{Hom}_0^{[n]}(\G,\C)$ spanned by maps of the form
\er{frm} for $0\lqs m\lqs n-1$. It follows from the  identity
\er{iabc} for changing the base point that $\operatorname{H}^{[n]}(\G,\C)$ is independent of $a$.

The next result shows that including more iterated integrals in the product and changing endpoints does not take you out of the space $\operatorname{H}^{[n]}(\G,\C)$, provided that the number of cusp forms used is $\lqs n-1$.

\begin{prop}
We have
\begin{equation} \label{jhp}
  \prod_{i=1}^u C_{a_i}^{\g b_i}(f_{i_1}, \dots ,f_{i_{r(i)}}) \prod_{j=1}^v \overline{C_{a'_j}^{\g b'_j}(f_{j_1}, \dots ,f_{j_{t(j)}})} \in \operatorname{H}^{[n]}(\G,\C)
\end{equation}
for $a_i,$ $b_i$, $a'_j,$ $b'_j$ in $\H \cup \{\mathtt{cusps}\}$ and $r(1)+\cdots+r(u)+t(1)+\cdots+t(v) \lqs n-1$.
\end{prop}
\begin{proof}
With  \er{iabc} we may rewrite \er{jhp} as a linear combination of terms of the same form as \er{jhp} but with every $a_i,$ $b_i$, $a'_j,$ $b'_j$ replaced by a fixed $a$.

We next introduce the {\em shuffle permutations}; see for example \cite[Lemma 1.1, (iv)]{DH}.
Let $sh(j,k)$ be the set of all permutations $\rho$ of $\{1,2,\dots,j+k\}$ such that
\begin{equation*}
  \rho(1)< \cdots <\rho(j) \quad \text{and} \quad  \rho(j+1)< \cdots <\rho(j+k).
\end{equation*}
Then  the {\em shuffle relation} implies
\begin{equation} \label{shuffle}
  C_a^{\g a}(f_1, \dots ,f_j)  C_a^{\g a}(f_{j+1}, \dots ,f_k) = \sum_{\rho \in sh(j,k)}  C_a^{\g a}(f_{\rho(1)}, f_{\rho(2)}, \dots ,f_{\rho(j+k)}).
\end{equation}
Therefore products of  two iterated integrals  may be expressed as a sum of single iterated integrals. Applying this repeatedly then gives \er{jhp} as a linear combination of terms of the form \er{frm} as desired.
\end{proof}

For the inclusion
\begin{equation}\label{true}
  \operatorname{H}^{[n]}(\G,\C) \subseteq \operatorname{Hom}_0^{[n]}(\G,\C),
\end{equation}
we clearly have equality when $n=1$ and also when $n=2$ by \cite[Prop. 2.1]{GOS}, for example.
Do we have equality in \er{true} for higher values of $n$? If not, how are the extra $n$th-order maps on the right  described?

{\small
\bibliography{sources}

\begin{thebibliography}{CDO02}

\bibitem[BD16]{BD}
Roelof Bruggeman and Nikolaos Diamantis.
\newblock Fourier coefficients of {E}isenstein series formed with modular
  symbols and their spectral decomposition.
\newblock {\em J. Number Theory}, 167:317--335, 2016.

\bibitem[CDO02]{cdos}
G.~Chinta, N.~Diamantis, and C.~O'Sullivan.
\newblock Second order modular forms.
\newblock {\em Acta Arith.}, 103(3):209--223, 2002.

\bibitem[CI13]{ChoieIhara}
YoungJu Choie and Kentaro Ihara.
\newblock Iterated period integrals and multiple {H}ecke {$L$}-functions.
\newblock {\em Manuscripta Math.}, 142(1-2):245--255, 2013.

\bibitem[DH13]{DH}
Anton Deitmar and Ivan Horozov.
\newblock Iterated integrals and higher order invariants.
\newblock {\em Canad. J. Math.}, 65(3):544--552, 2013.

\bibitem[DS08]{DS}
Nikolaos Diamantis and David Sim.
\newblock The classification of higher-order cusp forms.
\newblock {\em J. Reine Angew. Math.}, 622:121--153, 2008.

\bibitem[GO03]{GOS}
Dorian Goldfeld and Cormac O'Sullivan.
\newblock Estimating additive character sums for {F}uchsian groups.
\newblock {\em Ramanujan J.}, 7(1-3):241--267, 2003.
\newblock Rankin memorial issues.

\bibitem[Gol99]{Go}
Dorian Goldfeld.
\newblock The distribution of modular symbols.
\newblock In {\em Number theory in progress, {V}ol. 2
  ({Z}akopane-{K}o\'scielisko, 1997)}, pages 849--865. de Gruyter, Berlin,
  1999.

\bibitem[Hor15]{Ho15}
Ivan Horozov.
\newblock Noncommutative {H}ilbert modular symbols.
\newblock {\em Algebra Number Theory}, 9(2):317--370, 2015.

\bibitem[IO09]{IO}
{\"O}zlem Imamo{\=g}lu and Cormac O'Sullivan.
\newblock Parabolic, hyperbolic and elliptic {P}oincar\'e series.
\newblock {\em Acta Arith.}, 139(3):199--228, 2009.

\bibitem[Iwa02]{IwSp}
Henryk Iwaniec.
\newblock {\em Spectral methods of automorphic forms}, volume~53 of {\em
  Graduate Studies in Mathematics}.
\newblock American Mathematical Society, Providence, RI, second edition, 2002.

\bibitem[JO08]{JOS08}
Jay Jorgenson and Cormac O'Sullivan.
\newblock Unipotent vector bundles and higher-order non holomorphic
  {E}isenstein series.
\newblock {\em J. Th\'eor. Nombres Bordeaux}, 20(1):131--163, 2008.

\bibitem[JOS]{JOSS}
Jay Jorgenson, Cormac O'Sullivan, and Lejla Smajlovic.
\newblock Modular {D}edekind symbols associated to {F}uchsian groups and
  higher-order {E}isenstein series.
\newblock Submitted. Available on the arXiv.

\bibitem[KZ03]{kz}
Peter Kleban and Don Zagier.
\newblock Crossing probabilities and modular forms.
\newblock {\em J. Statist. Phys.}, 113(3-4):431--454, 2003.

\bibitem[Man72]{M1}
Ju.~I. Manin.
\newblock Parabolic points and zeta functions of modular curves.
\newblock {\em Izv. Akad. Nauk SSSR Ser. Mat.}, 36:19--66, 1972.

\bibitem[Man06]{M2}
Yuri~I. Manin.
\newblock Iterated integrals of modular forms and noncommutative modular
  symbols.
\newblock In {\em Algebraic geometry and number theory}, volume 253 of {\em
  Progr. Math.}, pages 565--597. Birkh\"auser Boston, Boston, MA, 2006.

\bibitem[Man14]{M3}
Yuri~I. Manin.
\newblock Non-commutative generalized {D}edekind symbols.
\newblock {\em Pure Appl. Math. Q.}, 10(2):245--258, 2014.

\bibitem[O'S00]{OS}
Cormac O'Sullivan.
\newblock Properties of {E}isenstein series formed with modular symbols.
\newblock {\em J. Reine Angew. Math.}, 518:163--186, 2000.

\bibitem[Pet02]{Pe}
Yiannis~N. Petridis.
\newblock Spectral deformations and {E}isenstein series associated with modular
  symbols.
\newblock {\em Int. Math. Res. Not.}, (19):991--1006, 2002.

\bibitem[PR04]{PeRi}
Y.~N. Petridis and M.~S. Risager.
\newblock Modular symbols have a normal distribution.
\newblock {\em Geom. Funct. Anal.}, 14(5):1013--1043, 2004.

\bibitem[PR18]{PR18}
Yiannis~N. Petridis and Morten~S. Risager.
\newblock Arithmetic statistics of modular symbols.
\newblock {\em Invent. Math.}, 212(3):997--1053, 2018.

\bibitem[Ris03]{Ri}
M.~S. Risager.
\newblock {\em Automorphic forms and modular symbols}.
\newblock PhD thesis, University of Aarhus, 2003.

\end{thebibliography}


\vskip 5mm

\textsc{Department of Mathematics,
 The City College of New York,
 New York, NY 10031, USA
 }

\textit{E-mail address:} \texttt{gchinta@ccny.cuny.edu}

\vskip 3mm

\textsc{Department of Mathematics, City University of New York, Bronx
  Community College,}

\textsc{New York, NY 10453, USA}

\textit{E-mail address:} \texttt{ivan.horozov@bcc.cuny.edu}

\vskip 3mm

\textsc{Department of Mathematics, The CUNY Graduate Center, New
   York, NY 10016, USA}

\textit{E-mail address:} \texttt{cosullivan@gc.cuny.edu}

}

\end{document}